\documentclass{birkjour}
\usepackage{hyperref}
\hypersetup{nesting=true,debug=true,naturalnames=true}
\usepackage{graphicx,amssymb,upref}

\usepackage{amsmath,amsthm}
\usepackage{amsfonts}
\usepackage{latexsym}
\usepackage{amsbsy}
\usepackage{amssymb,mathscinet}
\usepackage{color}

\renewcommand{\theequation}{\thesection.\arabic{equation}}

\numberwithin{equation}{section}

\usepackage{amsfonts,amssymb,amscd,amsmath,enumerate,verbatim,calc,enumerate}
\theoremstyle{plain}
\newtheorem{theorem}{Theorem}[section]
\newtheorem{corollary}[theorem]{Corollary}
\newtheorem{lemma}[theorem]{Lemma}

\newtheorem{observation}[theorem]{Observation}
\newtheorem{definition}[theorem]{Definition}

\newcommand{\raro}{\rightarrow}

\newcommand{\be}{\mathbb E}

\newcommand{\ot}{\otimes}
\newcommand {\id} {{\textrm{id}}}
\newcommand{\wt}{\widetilde}
\newcommand{\wT}{\wt{T}}

\newcommand{\wV}{\wt{V}}

\begin{document}

%
%
%
%
%
%
%
%
%

	\title[Completely contractive covariant representations]{Completely contractive covariant representations of product system over $\mathbb N^2_0$}

\author[Saini]{Dimple Saini}
\address{Department of Mathematics, \\G. L. Bajaj Institute of Technology and Management,\\ Greater Noida, India}
\email{dimple92.saini@gmail.com, dimple@glbitm.ac.in}



\date{January 31, 2026}

	\begin{abstract}
	A pure completely contractive covariant representations of a $C^*$-correspondence dilate to a pure isometric covariant representations due to Muhly-Solel. More specifically, we are curious about the following question for pairs: Does a pure completely contractive covariant representations of a product system dilate to a pure isometric covariant representations of a product system? The goal of this study is to find a pure completely contractive covariant representations of a product system that provide an affirmative answer for the previous question.  
\end{abstract}

\maketitle
	\section{Introduction}
According to Wold's fundamental theorem (cf. \cite{W1938}), every isometric operator $V$ on a Hilbert space $\mathcal{H}$ is unitarily equivalent to $M_z \oplus U,$ where $U$ is unitary on $\mathcal{H}_u:=\cap_{n\ge 0}V^n\mathcal{H}$ and $M_z$ is the shift on $\mathcal{D}$-valued Hardy space $H^2_{\mathcal{D}}(\mathbb{D})$ over unit disc $\mathbb{D},$ for some Hilbert space $\mathcal{D}.$ A contraction $T$ on a Hilbert space $\mathcal{H}$ is said to be pure if $T^{*m}\to 0$ in strong operator topology (SOT), (i.e., $\|T^{*m}h\| \rightarrow 0$ as $m \rightarrow \infty$ for all $h\in \mathcal{H}$). Sz.-Nagy and Foias theorem \cite{NFBK10} proved that every pure contraction $T$ on a Hilbert space $\mathcal{H}$ is unitarily equivalent to an operator of the form $P_{\mathcal{Q}}M_{z}|_{\mathcal{Q}},$ where $\mathcal{Q}$ is a closed $M_z^*$-invariant subspace of $H^2_{\mathcal{D}}(\mathbb{D})$ and $P_{\mathcal{Q}}$ is an orthogonal projection of $H^2_{\mathcal{D}}(\mathbb{D})$ onto $\mathcal{Q}.$ Extending Nagy and Foias dilation theorem, Ando \cite{AT63} established the existence of isometric dilations for pairs of commuting contractions. Later, Das and Sarkar \cite{DS17} presented an explicit construction of such dilations. {Parrott \cite{P70} gave a counterexample showing that for $n\ge 3,$ $n$-tuples of commuting contractions do not generally admit isometric dilations. This striking limitation highlights the intrinsic difficulties of multivariable dilation theory and motivates a more nuanced investigation into the specific conditions under which such dilations might still be possible. This work also has a conceptual connection with the operator-theoretic foundations of quantum theory. In the operator-algebraic approach to quantum systems, Hilbert space operators and representations of operator algebras furnish a natural framework for describing observables, symmetries, and dynamical processes. In this context, dilation theory provides a systematic method for embedding contractive operators into a broader isometric structure on an enlarged Hilbert space.}

Berger, Coburn, and Lebow \cite{BCL78} discussed the idea of a pure pair of commuting isometries, which is a significant advancement in the study of representation and Fredholm theory for $C^*$-algebras produced by commuting isometries. Berger, Coburn, and Lebow \cite{BCL78} studied the following result:
\begin{theorem}
	Assume $(V_1, V_2)$ to be a pair of commuting isometries on $\mathcal{H},$ and $V = V_1 V_2$ to be pure. Then there exist a unitary $U$ on a Hilbert space $\mathcal{E}$ and an orthogonal projection $P$ on $\mathcal{E}$ such that
	\[
	{\Psi}(z)=U(P^{\perp}+zP)
	\quad and \quad {\Phi}(z)=(P+zP^{\perp})U^* 
	\]
	are  inner functions   in $H^\infty_{B(\mathcal{E})}(\mathbb{D})$  and  the triple $(M_{\Phi}, M_{\Psi}, M_z)$ is unitarily equivalent to $(V_1, V_2, V).$
\end{theorem}

Das, Sarkar and Sarkar \cite{DSS17} obtained a similar result for pure pairs of commuting contractions in the following way:
\begin{theorem}
	Let $T$ be a pure contraction and $(T_1,T_2)$ be a pair of commuting
	contractions on a Hilbert space $\mathcal{H}.$ Let $\mathcal{Q}$ be the Sz.-Nagy and Foias representation of $T.$ Then the following are equivalent:
	\begin{enumerate}
		\item $T=T_1T_2;$
		\item There exist a triple $(\mathcal{E}, U, P )$ and a joint $(M_z^*,M_{\Phi}^*,M_{\Psi}^*)$-invariant subspace $\mathcal{Q}'$ of $H_{\mathcal{E}}^2(\mathbb{D})$
		such that $$T_1\cong P_{\mathcal{Q}'}M_{\Phi}|_{\mathcal{Q}'},T_2\cong P_{\mathcal{Q}'}M_{\Psi}|_{\mathcal{Q}'},T\cong P_{\mathcal{Q}'}M_{z}|_{\mathcal{Q}'}$$
		$$\quad and \quad M_{\Phi}M_{\Psi}=M_{\Psi}M_{\Phi}=M_{z}.$$
		\item There exist $B(\mathcal{D})$-valued polynomials $\phi$ and $\psi$ of degree $\le 1$ such that $\mathcal{Q}$ is a joint
		$(M_{\phi}^*,M_{\psi}^*)$-invariant subspace, $$P_{\mathcal{Q}}M_{\phi\psi}|_{\mathcal{Q}}=P_{\mathcal{Q}}M_{\psi\phi}|_{\mathcal{Q}}=P_{\mathcal{Q}}M_{z}|_{\mathcal{Q}} \quad and \quad (T_1,T_2)\cong (P_{\mathcal{Q}}M_{\phi}|_{\mathcal{Q}},P_{\mathcal{Q}}M_{\psi}|_{\mathcal{Q}}).$$
	\end{enumerate}
\end{theorem}

{Operator theory and its inseparable counterpart, function theory, are well-established disciplines with a long history. The overarching field, in keeping with traditional questions, is undergoing continuous development while also witnessing the generalizations of classical concepts in various settings. One of the generalized ideas of classical operator and function theory is the notion of Hilbert $C^*$-modules in terms of $C^*$-correspondences (over a $C^*$-algebra $\mathcal{B}$) and covariant representations. If $\mathcal{B}=\mathbb{C}$ and $C^*$-correspondence also equal to $\mathbb{C}^n,$ then the setting of isometric covariant representation reduces to row isometry. Therefore to work on this far general setting we need to define the relevant techniques and come up with some new tools to prove results.}

Covariant representations of product system are one of the most advanced theories in operator theory and operator algebras that offer a cohesive method for examining both commuting and non-commuting tuples of operators on Hilbert spaces. For $C^*$-correspondences, Muhly and Solel \cite{MS99} proved Wold decomposition for the isometric covariant representations. Muhly and Solel \cite{MS98} introduced isometric dilation for completely contractive covariant representations of $C^*$-correspondences.
{Rohilla and Saini \cite{ARDS1} 
studied representation of product systems over the semigroup $\mathbb{Z}^2_+$ and proved that a completely contractive representation can be dilated to an isometric representation. Barik, Bhattacharjee and Solel \cite{BMS25} discussed representations of product systems over the semigroup $\mathbb{Z}^n_+$ and showed that, under certain pureness and Szego positivity conditions, a completely contractive representation can be dilated to an isometric representation.} 

 The goal of this study is to find an isometric dilation of pure completely contractive covariant representations of product system. Our work is based on an explicit version of the isometric dilation for pure pairs of commuting contractions obtained in \cite{DSS17, S21}. {The paper is organized as follows. In Sections $1$ and $2,$ we recall the basic framework, introduce the relevant definitions, and present several useful preliminary results. In Section $3,$ we establish that a pure contractive covariant representations always dilates to a pure isometric covariant representations of product system over the semigroup $\mathbb{Z}^2_+$. In Section $4,$ we investigate the factorization of completely contractive covariant representations. Finally, in Section $5,$ we demonstrate several further important properties of the unitary module map.}

\subsection{Preliminaries and Notations}
First, we'll review some basic concepts from \cite{L95, MS98, MS99, MR0355613}.	
Assume $\mathcal{B}$ to be a $C^*$-algebra and $E$ to be a {\it $C^*$-correspondence} over $\mathcal{B}$ with the left module map $\phi:\mathcal B\to \mathcal L(E)$ defined by $ a\xi:=\phi(a)\xi, a\in\mathcal{B}$ and $\xi\in E,$ where $\mathcal L(E)$ denotes the collection of all adjointable operators on $E.$ Throughout this article, we will use the following notations: $E$ denotes a $C^*$-correspondence over $\mathcal{B}$, $\mathcal{H}$ is a Hilbert space, $B(\mathcal{H})$ represents the collection of all bounded linear operators on $\mathcal{H},$  each $*$-homomorphism is nondegenerate and $\mathbb N_0 =\mathbb N\cup\{0\}.$

\begin{definition}
	Assume $T:E\to B(\mathcal H)$ to be a linear map and  $\sigma:\mathcal{B}\to B(\mathcal H)$ to be a representation. We say that the pair $(\sigma,T)$ is a {\rm completely contractive covariant representation} (or simply write, {\rm c.c.c. representation}) (cf. \cite{MS98}) of $E$ on $\mathcal H$  if $T$ is completely contractive and
	\[
	T(ax b)=\sigma(a)T(x)\sigma(b) \quad \quad (
	a,b\in\mathcal{B}~ \mbox{and}~ x\in E).
	\] 
	
\end{definition}

Due to Muhly and Solel \cite{MS98}, the following result is helpful to classify the completely contractive covariant representations of a $C^*$-correspondence.

\begin{lemma} \label{MSC}
	The function $(\sigma, T)\mapsto \widetilde{T}$ gives a bijection between the set of all c.c.c. representations $(\sigma, T)$
	and the set of all contractive linear maps $\widetilde{T}:E\otimes_{\sigma} \mathcal H\to \mathcal
	H$ defined by
	\[
	\widetilde{T}(x\otimes h)=T(x)h \quad \quad (
	h\in\mathcal H,x\in E),
	\] and satisfies $\widetilde{T}(\phi(a)\otimes I_{\mathcal
		H})=\sigma(a)\widetilde{T}$ for all $a\in\mathcal{B}$. Further, $\wT$ is {\rm isometry} (resp. {\rm co-isometry, partial isometry }) if and only if $(\sigma, T)$ is isometric (resp. {\rm co-isometric, partial isometric}).
\end{lemma}
Assume $E$ to be a $C^*$-correspondence over $\mathcal{B}.$ Then,
$E^{\otimes m}: =E\otimes_{\phi} \cdots \otimes_{\phi}E$ ($m\in\mathbb{N}_0$ times) is a $C^*$-correspondence over $\mathcal{B}$ with $E^{\ot 0}=\mathcal{B}$ and left action of $\mathcal{B}$ on $E^{\otimes m}$  is defined by  $$\phi^m(a)(x_1 \otimes \cdots \otimes x_m):=ax_1\otimes \cdots \otimes x_m \quad \quad \quad (x_i \in E, a\in\mathcal{B}).$$ 
For  $m\in \mathbb{N},$ define $\wT_m : E^{\ot m}\ot \mathcal{H} \to \mathcal{H}$ by 
$$\wT_m (x_1 \ot \cdots \ot x_m \ot h) = T (x_1) \cdots T(x_m) h \quad \quad \quad (x_i \in E, h \in \mathcal H).$$
The Fock module, $\mathcal{F}(E):= \bigoplus_{m \geq 0}E^{\otimes m},$ is a $C^*$ correspondence over $\mathcal{B},$ where the left module action of $\mathcal{B}$ on $\mathcal{F}(E)$ is given by $$\phi_{\infty}(a)(\oplus_{m \geq 0}x_m):=\oplus_{m \geq 0}a\cdot x_m , \:\: x_m \in E^{\otimes m}.$$
Let $\xi \in E,$ define the {\it creation operator} $T_{\xi}$ on $\mathcal{F}(E)$  by $$T_{\xi}(x):=\xi \otimes x, \:\: x \in E^{\otimes m}.$$
Let $\pi$ be a representation of $\mathcal{B}$ on $\mathcal{H}$. An isometric covariant representation $(\rho, S)$ of $E$ on a Hilbert space $\mathcal{F}(E)\otimes_{\pi}\mathcal{H}$ given by
\begin{align*}
	\rho(a):=\phi_{\infty}(a) \otimes I_{\mathcal{H}} \:\:, a \in \mathcal{B} \quad and \quad
	S(\xi):=T_{\xi}\otimes I_{\mathcal{H}} ,\:\: \xi \in E.
\end{align*}
We say that $(\rho,S)$ is an induced representation of $E$ (\cite{mR74}) induced  by $\pi$. Note that any isometric representation which is isomorphic to an induced representation is also induced.

\begin{definition}	
	A closed subspace $\mathcal{W}\subseteq\mathcal{H}$ is said to be {\rm invariant} under the covariant representation $(\sigma, T)$ $(resp. (\sigma,T)$-{\rm reducing}) if $\mathcal{W}$ is  $\sigma{(\mathcal{B})}$-invariant and (resp. both $\mathcal{W}^{\perp}, \mathcal{W}$) is invariant under each operator  $T(x)$ for all $x \in E.$ The restriction $(\sigma , T)|_{\mathcal{W}}$ gives a new representation of $E$ on $\mathcal{W}.$ We say that $\mathcal{W}$ is {\rm co-invariant} if $\wT^*{(\mathcal{W})}\subseteq E\ot \mathcal{W}.$
\end{definition}

Assume $(\sigma, T)$ to be a c.c.c. representation of $E$ on  $\mathcal{H},$ that is, $I_{\mathcal{H}}-\wT\wT^*\ge 0.$ We say that $(\sigma, T)$ is {\it pure} if $\wT_n^*\to 0$ in SOT. Now we define a positive operator $\Delta(T):=(I-\widetilde{T}{\widetilde{T}}^*)^{\frac{1}{2}}$ and the corresponding defect space $\mathcal{D}_{T}:=\overline{Im\Delta(T)}$. Since $$\wT\wT^*\sigma(a)=\wT(\phi(a)\ot I_{\mathcal{H}})\wT^*=\sigma(a)\wT\wT^*$$ for $a\in \mathcal{B},$ we get  $\Delta(T) \in \sigma(\mathcal{B})',$ and hence $\mathcal{D}_{T}$ is $\sigma(\mathcal{B})$-reducing. Define $(\sigma_0, T_0):=(\sigma ,T)|_{\mathcal{D}_{T}},$ for simplicity we write $(\sigma, T)$ instead of $(\sigma_0, T_0)$. The Poisson kernel of $(\sigma, T)$ \cite{MS98} is the operator $\Pi: \mathcal{H} \longrightarrow \mathcal{F}(E)\otimes_{\sigma}\mathcal{D}_{T}$ given by $$\Pi u:=\sum_{m\geq0}(I_{E^{\otimes m}}\otimes\Delta(T))\widetilde{T}_m^* u,\:\:\: u \in \mathcal{H}.$$

Now, we present an isometric dilation of c.c.c. representations of a $C^*$-correspondence due to Muhly and Solel \cite{MS98}. 	
\begin{theorem}\label{dilation}
	Assume $(\sigma, T)$ to be a c.c.c. representation of $E$ on $\mathcal{H},$ then the poisson kernel $\Pi$ is a contraction and
	\begin{align} \label{sid}
		(I_{E} \ot \Pi)\widetilde{T}^{*}=\widetilde{S}^{*}\Pi  \:\: \mbox{and} \:\:  	\Pi \sigma(a)=\rho(a)\Pi, \:\:\:a \in \mathcal{B}. 
	\end{align}
	Moreover, $\Pi$ is an isometry if and only if $(\sigma, T) $ is pure. The induced representation $(\rho,S)$ of $E$ induced by $\sigma|_{\mathcal{D}_{T}}$ is called {\rm isometric dilation} of $(\sigma, T).$
\end{theorem}

\section{Completely contractive covariant representations over product system}

A \emph{product system} $\be$ over $\mathbb{N}^2_0$ consists of a collection of $C^*$-correspondences $\{E_1,E_2\}$    (see \cite{F02, DTV2, S06, S08}) along with unitary isomorphisms $t_{j,i}: E_j \ot E_i \to E_i \ot E_j$ ($j>i$). Using this, we identify each correspondence $\be (m)$ for $m=(m_1,m_2) \in \mathbb N^2_0$ with $E_1^{\ot^{ m_1}}\ot E_2^{\ot^{m_2}}.$ We also set $t_{j,j} = \id_{E_j \ot E_j}$ and $t_{j,i} = t_{i,j}^{-1}$ when \(j < i\).

\begin{definition}
	Assume $\be$ to be a product system over $\mathbb N^2_0,$ $\sigma:\mathcal{B}\to B(\mathcal H)$ to be a representation and $T^{(i)}: E_i\to B(\mathcal H)$, $i\in I_2:=\{1,2\}$ to be the linear maps. Then the triple $(\sigma, T^{(1)}, T^{(2)})$ is said to be {\rm completely contractive covariant representation} of $\be$ on $\mathcal{H}$ if each pair $(\sigma, T^{(i)})$ is a c.c.c. representation of $E_i$ on $\mathcal{H}$ and satisfies
	\begin{equation} \label{rep} \wT^{(j)} (I_{E_j} \ot \wT^{(i)}) = \wT^{(i)} (I_{E_i} \ot \wT^{(j)}) (t_{j,i} \ot I_{\mathcal H}),\quad \quad  j,i\in I_2.
	\end{equation}
\end{definition}

We say that two such representations $(\sigma, T^{(1)}, T^{(2)})$ and $(\psi, V^{(1)}, V^{(2)})$ of $\be$ over $\mathbb N^2_0,$
respectively on Hilbert spaces $\mathcal H$ and $\mathcal K$, are {\it isomorphic} if there exists a unitary $U:\mathcal H \to
\mathcal K$ such that $U\sigma(a)=\psi(a)U$ and $V^{(i)}(x_i)U = U T^{(i)} (x_i)$ for all $a\in \mathcal{B}$ and $x_i\in E_i,i\in I_2.$

Assume $\pi $ to be a representation of a $C^*$-algebra $\mathcal{B}$ on a Hilbert space $\mathcal{K}$ and $\mathbb{E}$ to be a product system over $\mathbb{N}_0^2.$ Suppose that $\Theta: E_2 \to B(\mathcal{K}, \mathcal{F}(E_1) \otimes_{\pi} \mathcal{K})$ is a completely contractive  bi-module map, that is, $\Theta$ is  completely contractive  and $\Theta(a \eta b)=\rho(a)\Theta(\eta)\rho(b)$ for all $\eta \in E_2, a,b \in \mathcal{B}.$ Define a contractive linear  map $\widetilde{\Theta}: E_2 \otimes \mathcal{K} \to \mathcal{F}(E_1) \otimes_{\pi} \mathcal{K}$ by $$\widetilde{\Theta}(\eta \otimes h)=\Theta(\eta)h ~~\:\:\: \mbox{for}~~~~~~\:\:\: \eta\in E_2,h\in \mathcal{K},$$ and it satisfy the relation $\widetilde{\Theta}(\phi_2(a) \ot I_{ \mathcal{K}})=\rho(a)\widetilde{\Theta},a\in \mathcal{B},$ where $\phi_2$ is a module left action of $\mathcal{B}$ on $E_2.$ A completely contractive map $M_{\Theta} :E_2  \to B(\mathcal{F}(E_1) \otimes_{\pi} \mathcal{K})$ \cite{STV22} is defined by
\begin{align*}
	M_{\Theta}(\eta) (S_n(x_n) h)=\widetilde{S}_n(I_{E_1^{\otimes n}} \otimes \widetilde{\Theta})(t_{2,1}^{(1,n)} \otimes I_{\mathcal{K}})(\eta \otimes x_n \otimes h),
\end{align*}
where  $\eta \in E_2, x_n \in E_1^{\otimes n}, h\in \mathcal{K}, n \in \mathbb{N}_0,$ $t_{2,1}^{(1,n)}: E_2 \ot E_1^{\ot n}  \rightarrow E_1^{\ot n} \ot E_2 $ is a composition of the isomorphisms $\{t_{j,i}\: :\: 1 \leq j, i \leq 2\}$ and $(\rho, S)$ is the induced representation of $E_1$ induced by $\pi.$ Clearly, $(\rho,{M}_{\Theta})$ is a c.c.c. representation of $E_2$ on $\mathcal{F}(E_1) \otimes_{\pi} \mathcal{K},$ $M_{\Theta}(\eta)|_{\mathcal{K}}=\Theta(\eta)$ for each $\eta \in E_2$ and satisfy the relation
\begin{align*}
	M_{\Theta}(\eta)\left(\bigoplus_{n \in \mathbb{N}_0}x_n \otimes h_n\right)=\sum_{n \in \mathbb{N}_0}\widetilde{S}_n(I_{E_1^{\otimes n}} \otimes \widetilde{\Theta})(t_{2,1}^{(1,n)} \otimes I_{\mathcal{K}})(\eta \otimes x_n \otimes h_n)
\end{align*} 
for all $\eta \in E_2, x_n \in E_1^{\otimes n}, h_n \in \mathcal{K}.$  It is easy to verify that
$\widetilde{M}_{\Theta} (I_{E_2} \otimes \widetilde{S})=\widetilde{S} (I_{{E}_1} \otimes \widetilde{M}_{\Theta})(t_{2,1} \otimes I_{\mathcal{F}(E_1) \otimes \mathcal{K}}).$ Hence $(\rho, S, M_{\Theta})$ is a c.c.c. representation of $\mathbb{E}$ over $\mathbb{N}_0^2$ on $\mathcal{F}(E_1) \otimes _{\pi} \mathcal{K}.$ Note that $\Theta$ is an isometry if and only if $(\rho, M_\Theta)$ is an isometric covariant representation. In this paper, we focus on $\Theta$ such that $(\rho, M_\Theta)$ is a c.c.c. representation.

The concept of pure isometric representation of product system over $\mathbb{N}_0^2$ introduced in \cite{STV22} is an important development in the study of representation and Fredholm theory for $C^*$-algebras generated by commuting isometries. 
They showed that: Assume $(\sigma , V^{(1)},V^{(2)})$ to be a pure isometric covariant representation of $\mathbb{E}$ over $\mathbb{N}_0^2$ on $\mathcal{H},$ then the BCL- representation $(\rho,M_{{\Theta_1}}, M_{{\Theta_2}})$ \cite{STV22} of $(\sigma , V^{(1)},V^{(2)})$ is defined by 
$$\widetilde{\Theta}_1 = (P_{\mathcal{W}_1}^{\perp} + \widetilde{S}(I_{E_1\ot E_2}\ot P_{\mathcal{W}_1}))U'
$$ and $$
\widetilde{\Theta}_2 = (P_{\mathcal{W}_2}^{\perp} + \widetilde{S}(t_{2,1}\ot P_{\mathcal{W}_2}))U,
$$ where
$U'= \begin{bmatrix}
	\wV^{(1)}|_{E_1 \ot \mathcal{W}_2} & 0 \\
	0 & (I_{E_1}\ot \wV^{(2)*}|_{\wV^{(2)}(E_2 \ot \mathcal{W}_1)})
\end{bmatrix} : E_1 \ot \mathcal{W}_2 \oplus E_1 \ot \wV^{(2)}(E_2 \ot \mathcal{W}_1) \to 	\wV^{(1)}(E_1 \ot \mathcal{W}_2) \oplus E \ot \mathcal{W}_1$ and 

$U= \begin{bmatrix}
	\wV^{(2)}|_{E_2 \ot \mathcal{W}_1} & 0\\
	0 & (I_{E_2}\ot \wV^{(1)*}|_{\wV^{(1)}(E_1 \ot \mathcal{W}_2)})
\end{bmatrix} : E_2 \ot \mathcal{W}_1 \oplus E_2\ot \wV^{(1)}(E_1 \ot \mathcal{W}_2) \to \wV^{(2)}(E_2 \ot \mathcal{W}_1) \oplus E_2 \ot E_1 \ot \mathcal{W}_2$  are unitary isomorphisms.

In this paper we study similar result for completely contractive covariant representations of product system over $\mathbb{N}_0^2.$ 
Suppose that $(\sigma , T^{(1)},T^{(2)})$ is a c.c.c. representation of $\mathbb{E}$ over $\mathbb{N}_0^2$ on  $\mathcal{H}.$ Define $\widetilde{T}:=\widetilde{T}^{(1)}(I_{E_1} \otimes \widetilde{T}^{(2)}),$ then by Lemma \ref{MSC} $(\sigma, T)$ is also a c.c.c. representation of $E_1 \ot E_2$ on $\mathcal{H}.$ The representation $(\sigma , T^{(1)},T^{(2)})$ is called pure if $(\sigma, T)$ is pure. Note that $(\sigma , T^{(1)},T^{(2)}, T)$ is a c.c.c. representation of the product system  over $\mathbb{N}_0^3$ determined by the $C^*$-correspondences $\{E_1, E_2, E\}$ on $\mathcal{H}.$ Since
\begin{align}\label{BB1}
	&I_{\mathcal H}-\widetilde{T}^{(1)}\widetilde{T}^{(1)^*}+\widetilde{T}^{(1)}(I_{E_1}\otimes(I_{\mathcal{H}}-\widetilde{T}^{(2)}\widetilde{T}^{(2)^*}))\widetilde{T}^{(1)^*}=	I_{\mathcal H}-{\widetilde T}{\widetilde T^*}\\& \nonumber=I_{\mathcal H}-\widetilde{T}^{(2)}\widetilde{T}^{(2)^*}+
	\quad \widetilde{T}^{(2)}(I_{E_2}\otimes(I_{\mathcal{H}}-\widetilde{T}^{(1)}\widetilde{T}^{(1)^*}))\widetilde{T}^{(2)^*}.
\end{align}
This implies that
\begin{align*}
	\|\Delta(T^{(1)})h\|^2+\|(I_{E_1} \ot \Delta(T^{(2)}))\wT^{(1)^*}h\|^2
	=\\ \|(I_{E_2} \ot \Delta(T^{(1)}))\wT^{(2)^*}h\|^2+\|\Delta(T^{(2)})h\|^2,
\end{align*} for all $h\in \mathcal{H}.$
Thus \begin{align*}
	U:\{\Delta(T^{(1)})h\oplus (I_{E_1} \ot \Delta(T^{(2)}))\wT^{(1)^*}h: h\in \mathcal{H}\}\to \\ \{(I_{E_2} \ot \Delta(T^{(1)}))\wT^{(2)^*}h\oplus \Delta(T^{(2)})h: h\in \mathcal{H}\}
\end{align*}
 define by
\begin{align}\label{equation1}
\nonumber	U(\Delta(T^{(1)})h,  (I_{E_1} \ot \Delta(T^{(2)}))\wT^{(1)^*}h)=\\((I_{E_2} \ot \Delta(T^{(1)}))\wT^{(2)^*}h, \Delta(T^{(2)})h) \quad h\in \mathcal H \end{align} is an isometric module map (see \cite{ARDS1}). If  $$ \dim ((E_2 \ot \mathcal{D}_{*, T^{(1)}}) \oplus  \mathcal{D}_{*, T^{(2)}} )= \dim (\mathcal{D}_{*, T^{(1)}} \oplus  (E_1 \ot\mathcal{D}_{*, T^{(2)}}))< \infty,$$ then $U$ extends to a unitary module map.

Suppose that $(\sigma , T^{(1)},T^{(2)})$ is a c.c.c. representation of $\mathbb{E}$ over $\mathbb{N}_0^2$ on  $\mathcal{H}.$ If $(\sigma , T^{(1)})$ is pure, then by Theorem \ref{dilation}, we get an isometry $\Pi: \mathcal{H} \longrightarrow \mathcal{F}(E_1)\otimes\mathcal{D}_{ T^{(1)}}$ such that 
\begin{align}
	\rho(a)\Pi=\Pi \sigma(a) \:\: \mbox{and} \:\: (I_{E_1} \ot \Pi)\widetilde{T}^{(1)^*}=\widetilde{S}^{\mathcal{D}^*}\Pi, \:\:\:a \in \mathcal{B}, 
\end{align} where $(\rho, S^{\mathcal{D}})$ is an induced representaion of $E_1$ induced by $\sigma|_{\mathcal{D}_{T^{(1)}}}.$ 

Let $\pi$ be a representation of $\mathcal{B}$ on a Hilbert space $\mathcal{K}$, $(\rho', S )$ be the induced representaion of $E_1$ induced by $\pi$ and $V\in B({\mathcal{D}_{ {T^{(1)}}}},\mathcal{K})$ be an isometric module map. Define an isometry $\Pi_V: \mathcal{H} \longrightarrow \mathcal{F}(E_1)\otimes\mathcal{K}$  by
$\Pi_V:=(I_{\mathcal{F}(E_1)}\ot V)\Pi,$ then
\begin{align*} \rho'(a)\Pi_V=\Pi_V \sigma(a)    \quad and \quad  (I_{E_1}\ot \Pi_V)\wT^{(1)^*} =\widetilde{S}^{*}\Pi_V,
\end{align*}  for every $a\in \mathcal{B}.$ We conclude that $(\rho', S)$ is an isometric dilation of $(\sigma,T^{(1)}).$ Moreover, if $\mathcal{Q}:=\Pi_V\mathcal{H},$ then $\widetilde{S}^{*}(\mathcal{Q})\subseteq E_1\ot \mathcal{Q},$ that is, $\mathcal{Q}$ is co-invariant for $(\rho',S).$ Hence the representations $(\sigma, T^{(1)})$ and $(\pi',X)$ of ${E_1}$ acting on the Hilbert spaces $\mathcal{H}$ and $\mathcal{Q},$ respectively, are isomorphic, where
$\widetilde{X}= P_{\mathcal{Q}}\widetilde{S}|_{E_1\ot \mathcal{Q}}$ and $\pi'=\rho'|_{\mathcal{Q}}.$

\begin{theorem}\label{identity}
	Using the above notations, let \begin{equation}
		\label{unitary} U = \begin{bmatrix}A&B\\C&0\end{bmatrix}:
		\mathcal{K} \oplus  (E_1 \ot\mathcal{D}_{ T^{(2)}})\to (E_2 \ot \mathcal{K}) \oplus  \mathcal{D}_{ T^{(2)}},
	\end{equation} be a unitary module map such that $$U(V\Delta(T^{(1)})h,(I_{E_1} \ot \Delta(T^{(2)}))\wT^{(1)^*}h)=((I_{E_2} \ot V\Delta(T^{(1)}))\wT^{(2)^*}h,\Delta(T^{(2)})h)$$ for all $h\in \mathcal{H}.$ Then there exists a contractive linear map $\widetilde{\Phi}:E_2\ot \mathcal{K}\to \mathcal{F}(E_1)\ot \mathcal{K}$ such that $$ (I_{E_2}\ot \Pi_V)\wT^{(2)^*}=\widetilde{M}_{\Phi}^*\Pi_V,$$ where $\widetilde{\Phi}=A^*+(I_{E_1}\ot C^*)B^*.$
\end{theorem}
\begin{proof}
	Let $h \in \mathcal H$ we have
	\[\begin{bmatrix}A&B\\C&0\end{bmatrix} \begin{bmatrix} V\Delta(T^{(1)})h\\ (I_{E_1} \ot \Delta(T^{(2)}))\wT^{(1)^*}h \end{bmatrix}= \begin{bmatrix} (I_{E_2} \ot V\Delta(T^{(1)}))\wT^{(2)^*}h \\  \Delta(T^{(2)})h\end{bmatrix},\]
	that is,
	\begin{equation} \label{one} (I_{E_2} \ot V\Delta(T^{(1)}))\wT^{(2)^*}h= AV\Delta(T^{(1)})h+
		B(I_{E_1} \ot \Delta(T^{(2)}))\wT^{(1)^*}h
	\end{equation}
	and $\Delta(T^{(2)})h= CV\Delta(T^{(1)})h.$	It follows that
	\begin{equation*}\label{two}
		(I_{E_1}\ot \Delta(T^{(2)}))\wT^{(1)^*}h= (I_{E_1}\ot CV\Delta(T^{(1)}))\wT^{(1)^*}h,
	\end{equation*}
	From Equation (\ref{one}) we get
	\[\begin{split}
		(I_{E_2} \ot V\Delta(T^{(1)}))\wT^{(2)^*}h&= AV\Delta(T^{(1)})h+
		B(I_{E_1} \ot \Delta(T^{(2)}))\wT^{(1)^*}h \\ & =  AV\Delta(T^{(1)})h+
		B(I_{E_1}\ot CV\Delta(T^{(1)}))\wT^{(1)^*}h.
	\end{split}\]  Let $h\in \mathcal{H},\eta_n\in E_2\ot E_1^{\ot n}\ot \mathcal{K}$ and $n\in \mathbb{N}_{0}$ we obtain
	\begin{align*}
		&\langle (I_{E_1^{\ot n}}\ot (I_{E_2} \ot V\Delta(T^{(1)}))\wT^{(2)^*})\wT_n^{(1)^*}h,(t_{2,1}^{(1,n)}\ot I_{\mathcal{K}})\eta_n \rangle\\&=\langle (I_{E_1^{\ot n}}\ot AV\Delta(T^{(1)}))\wT_n^{(1)^*}h,(t_{2,1}^{(1,n)}\ot I_{\mathcal{K}})\eta_n \rangle \\& \quad\quad +\langle (I_{E_1^{\ot n}}\ot B(I_{E_1}\ot CV\Delta(T^{(1)}))\wT^{(1)^*})\wT_n^{(1)^*}h,(t_{2,1}^{(1,n)}\ot I_{\mathcal{K}})\eta_n \rangle \\&=\langle (I_{E_1^{\ot n}}\ot V\Delta(T^{(1)}))\wT_n^{(1)^*}h,(I_{E_1^{\ot n}}\ot A^*)(t_{2,1}^{(1,n)}\ot I_{\mathcal{K}})\eta_n \rangle \\& \quad\quad +\langle (I_{E_1^{\ot n}}\ot (I_{E_1}\ot V\Delta(T^{(1)})))\wT_{n+1}^{(1)^*}h,(I_{E_1^{\ot n}}\ot (I_{E_1}\ot C^*)B^*)(t_{2,1}^{(1,n)}\ot I_{\mathcal{K}})\eta_n \rangle.
	\end{align*} Define a contractive linear map $\widetilde{\Phi}:E_2\ot \mathcal{K}\to \mathcal{F}(E_1)\ot \mathcal{K}$ by $\widetilde{\Phi}=A^*+(I_{E_1}\ot C^*)B^*,$ then  \begin{align*}
		&\langle \widetilde{M}_{\Phi}^*\Pi_V h,\eta_n \rangle =\langle \Pi_V h,\widetilde{M}_{\Phi}\eta_n \rangle=\sum_{m \in \mathbb{N}_{0}}\langle (I_{\mathcal{F}(E_1)}\ot V)(I_{E_1^{\otimes m}}\otimes\Delta(T^{(1)}))\widetilde{T}_m^{(1)^*} h,\widetilde{M}_{\Phi}\eta_n\rangle\\&=\sum_{m \in \mathbb{N}_{0}}\langle (I_{E_1^{\otimes m}}\otimes V\Delta(T^{(1)}))\widetilde{T}_m^{(1)^*} h,\widetilde{S}_n(I_{E_1^{\otimes n}} \otimes \widetilde{\Phi})(t_{2,1}^{(1,n)} \otimes I_{\mathcal{K}})\eta_n\rangle\\&=\sum_{m \in \mathbb{N}_{0}}\langle (I_{E_1^{\otimes m}}\otimes V \Delta(T^{(1)}))\widetilde{T}_m^{(1)^*} h,(I_{E_1^{\otimes n}}\otimes (A^*+(I_{E_1}\ot C^*)B^*))(t_{2,1}^{(1,n)} \otimes I_{\mathcal{K}})\eta_n\rangle  \\&= \langle (I_{E_1^{\otimes n}}\otimes V \Delta(T^{(1)}))\widetilde{T}_n^{(1)^*} h,(I_{E_1^{\otimes n}}\otimes A^*)(t_{2,1}^{(1,n)} \otimes I_{\mathcal{K}})\eta_n\rangle \\&\quad\quad  + \langle (I_{E_1^{\otimes n+1}}\otimes V \Delta(T^{(1)}))\widetilde{T}_{n+1}^* h,(I_{E_1^{\otimes n}}\otimes (I_{E_1}\ot C^*)B^*)(t_{2,1}^{(1,n)} \otimes I_{\mathcal{K}})\eta_n\rangle.
	\end{align*} Therefore $\langle \widetilde{M}_{\Phi}^*\Pi_V h,\eta_n \rangle=\langle (I_{E_1^{\ot n}}\ot (I_{E_2} \ot V\Delta(T^{(1)}))\wT^{(2)^*})\wT_n^{(1)^*}h,(t_{2,1}^{(1,n)}\ot I_{\mathcal{K}})\eta_n \rangle.$ On the other hand, since \begin{align*}
		&\langle (I_{E_2}\ot \Pi_V)\wT^{(2)^*}h,\eta_n  \rangle=\langle (I_{E_2}\ot (I_{\mathcal{F}(E_1)}\ot V)\Pi)\wT^{(2)^*}h,\eta_n  \rangle\\&=\sum_{m \in \mathbb{N}_{0}}\langle (I_{E_2}\ot (I_{E_1^{\ot m}}\ot V\Delta(T^{(1)}))\wT_m^{(1)^*})\wT^{(2)^*}h,\eta_n  \rangle \\&=\langle (I_{E_2}\ot (I_{E_1^{\ot n}}\ot V\Delta(T^{(1)}))\wT_n^{(1)^*})\wT^{(2)^*}h,\eta_n  \rangle \\&=\langle (I_{E_2}\ot I_{E_1^{\ot n}}\ot V\Delta(T^{(1)}))(I_{E_2}\ot \wT_n^{(1)^*})\wT^{(2)^*}h,\eta_n  \rangle \\&=\langle (I_{E_2}\ot I_{E_1^{\ot n}}\ot V\Delta(T^{(1)}))(t_{1,2}^{(n,1)}\ot I_{\mathcal{H}})(I_{E_1^{\ot n}}\ot \wT^{(2)^*})\wT_n^{(1)^*}h,\eta_n  \rangle \\&=\langle (t_{1,2}^{(n,1)}\ot I_{\mathcal{K}})(I_{E_1^{\ot n}}\ot (I_{E_2}\ot  V\Delta(T^{(1)}))\wT^{(2)^*})\wT_n^{(1)^*}h,\eta_n  \rangle \\&=\langle (I_{E_1^{\ot n}}\ot (I_{E_2}\ot  V\Delta(T^{(1)}))\wT^{(2)^*})\wT_n^{(1)^*}h,(t_{2,1}^{(1,n)}\ot I_{\mathcal{K}})\eta_n  \rangle ,
	\end{align*}we get $(I_{E_2}\ot \Pi_V)\wT^{(2)^*}=\widetilde{M}_{\Phi}^*\Pi_V.$ 
\end{proof}

\begin{observation}
	\begin{enumerate}
		\item In the setting of Theorem \ref{identity}, if $\mathcal{Q}=ran(\Pi_V),$ then $$ \widetilde{S}^{*}(\mathcal{Q})\subseteq E_1\ot \mathcal{Q},~~ \rho'(a)\mathcal{Q}\subseteq \mathcal{Q}\quad \mbox{and}\quad \widetilde{M}_{\Phi}^*(\mathcal{Q})\subseteq E_2\ot \mathcal{Q} \quad \mbox{for}\quad a\in \mathcal{B},$$ that is, $\mathcal{Q}$ is co-invariant for $(\rho',S,M_{\Phi}).$
		\item The representations $(\sigma, T^{(1)}, T^{(2)})$ and $(\pi',X,Y)$ of $\mathbb{E}$ over $\mathbb{N}_0^2$ on the Hilbert spaces $\mathcal{H}$ and $\mathcal{Q},$ respectively, are isomorphic, where $\pi'=\rho'|_{\mathcal{Q}},$
		$\widetilde{X}= P_{\mathcal{Q}}\widetilde{S}|_{E_1\ot \mathcal{Q}}$ and $\widetilde{Y}= P_{\mathcal{Q}}\widetilde{M}_{\Phi}|_{E_2\ot \mathcal{Q}}.$
	\end{enumerate}
\end{observation}

\section{Dilating to pure isometric covariant representations}
We now demonstrate that a pure contractive covariant representation can be dilated to a pure isometric covariant representation. First, we explicitly outline the construction of these dilations in the context of finite-dimensional defect spaces.
\begin{theorem}\label{JJ3}
	Assume $(\sigma, T^{(1)}, T^{(2)})$ to be a pure c.c.c. representation  of $\be$ over $\mathbb N^2_0$ on $\mathcal H$ such that $\mbox{dim~} (\mathcal{D}_{T^{(1)}} \oplus (E_1 \ot\mathcal{D}_{ T^{(2)}}))=\mbox{dim~} (\mathcal{D}_{ T^{(2)}}\oplus (E_2\ot \mathcal{D}_{ T^{(1)}}))< \infty.$ Then $(\sigma, T^{(1)}, T^{(2)})$ dilates to a pure isometric covariant representation.
\end{theorem}
\begin{proof}
	Since $(\sigma, T)$ is pure c.c.c. representation of $E:=E_1 \ot E_2$ on $\mathcal H,$ where $\widetilde{T}=\widetilde{T}^{(1)}(I_{E_1} \otimes \widetilde{T}^{(2)}).$ Let $\Pi: \mathcal{H} \longrightarrow \mathcal{F}(E)\otimes\mathcal{D}_{T}$ be the isometric dilation of $(\sigma,T)$ (see Theorem \ref{dilation}). Let $\mathcal K= \mathcal{D}_{T^{(2)}} \oplus(E_2 \ot\mathcal{D}_{T^{(1)}}),$ from Equation (\ref{BB1}), we define an isometric module map $V:\mathcal{D}_{T}\to  \mathcal K$ by
	$$V(\Delta(T)h):=(\Delta(T^{(2)})h,  (I_{E_2} \ot \Delta(T^{(1)}))\wT^{(2)^*}h)$$ for $h\in \mathcal{H}.$ Consequently, $\Pi_V:=(I_{\mathcal{F}(E)}\ot V)\Pi: \mathcal{H} \longrightarrow \mathcal{F}(E)\otimes\mathcal{K}$ is an isometric dilation of $(\sigma,T).$ Thus $\wT\cong P_{\mathcal{Q}}\widetilde{S}|_{E\ot \mathcal{Q}}$, where $\mathcal{Q}=\Pi_V\mathcal{H},$ $(\rho,S)$ is an induced representation of $E$ and $\widetilde{S}^{*}(\mathcal{Q})\subseteq E\ot \mathcal{Q}$ (see Theorem \ref{identity}). 
	Since $$(I_{E_2}\ot V\Delta(T))\wT^{(2)^*}h=((I_{E_2}\ot \Delta(T^{(2)}))\wT^{(2)^*}h,(I_{E_2^{\ot 2}}\ot \Delta(T^{(1)}))\wT_2^{(2)^*}h),\quad  h\in \mathcal{H}.$$
	Let $j_1 : \mathcal{D}_{T^{(2)}} \raro \mathcal{K}$ and $j_2 :  E_2 \ot\mathcal{D}_{ T^{(1)}} \raro \mathcal{K}$ be
	the inclusion maps, defined by
	\[
	j_1(h_1) := (h_1,0)\quad  \mbox{and} \quad j_2 (h_2) := (0,h_2)
	\quad (h_1 \in \mathcal{D}_{T^{(2)}}, h_2 \in  E_2 \ot\mathcal{D}_{ T^{(1)}}).
	\]
	Then $P : = j_1 j_1^* \in B(\mathcal{K})$ is an orthogonal
	projection of $\mathcal{K}$ onto $ \mathcal{D}_{ T^{(2)}}$, i.e.,
	\[
	P(h_1,h_2) = (h_1,0) \quad \quad \quad (h_1,h_2) \in \mathcal{K}.
	\]
	Thus, $j_2 j_2^* = P^\perp$ is the orthogonal projection of $\mathcal{K}$
	onto $E_2 \ot\mathcal{D}_{ T^{(1)}}.$ Since $\mbox{dim~} (\mathcal{D}_{T^{(1)}} \oplus (E_1 \ot\mathcal{D}_{ T^{(2)}}))< \infty$ and $\mbox{dim~} (\mathcal{D}_{ T^{(2)}}\oplus (E_2\ot \mathcal{D}_{ T^{(1)}}))< \infty.$ Define a
	unitary module map $U:E_2\ot\mathcal{K}\to E_2\ot E_1 \ot \mathcal{D}_{ T^{(2)}}\oplus E_2\ot \mathcal{D}_{ T^{(1)}}$ by
	\begin{align*}U((I_{E_2}\ot \Delta(T^{(2)}))\wT^{(2)^*}h,(I_{E_2^{\ot 2}} \ot \Delta(T^{(1)}))\wT_2^{(2)^*}h):=\\((t_{1,2}\ot \Delta(T^{(2)}))\wT^{*}h,(I_{E_2}\ot \Delta(T^{(1)}))\wT^{(2)^*}h).\end{align*} 
	Then
	\[
	U_1=
	\begin{bmatrix}
		U^*P^{\perp}& U^*(t_{1,2}\ot j_1)\\
		j_1^*& 0
	\end{bmatrix}: \mathcal{K} \oplus  (E \ot\mathcal{D}_{ T^{(2)}})\to (E_2 \ot \mathcal{K}) \oplus  \mathcal{D}_{ T^{(2)}}
	\] is a unitary module map. For $h\in \mathcal{H}$ we have
	\begin{align*}
		&U_1(V\Delta(T)h,(I_{E}\ot \Delta(T^{(2)}))\wT^{*}h)\\&=U_1(\Delta(T^{(2)})h,  (I_{E_2} \ot \Delta(T^{(1)}))\wT^{(2)^*}h,(I_{E}\ot \Delta(T^{(2)}))\wT^{*}h)\\&=(U^*((t_{1,2}\ot \Delta(T^{(2)}))\wT^{*}h,(I_{E_2}\ot \Delta(T^{(1)}))\wT^{(2)^*}h),\Delta(T^{(2)})h)\\&=((I_{E_2}\ot \Delta(T^{(2)}))\wT^{(2)^*}h,(I_{E_2^{\ot 2}} \ot \Delta(T^{(1)}))\wT_2^{(2)^*}h,\Delta(T^{(2)})h)\\&=((I_{E_2}\ot V\Delta(T))\wT^{(2)^*}h,\Delta(T^{(2)})h).
	\end{align*} From Theorem \ref{identity}, there exists a contractive linear map $\widetilde{\Phi}_2:E_2 \ot \mathcal{K}\to \mathcal{F}(E)\ot \mathcal{K}$ such that $$(I_{E_2}\ot \Pi_V)\wT^{(2)^*}=\widetilde{M}_{{\Phi}_2}^*\Pi_V,$$ where $\widetilde{\Phi}_2=P^{\perp}U+(I_{E}\ot j_1)(t_{2,1}\otimes j_1^*)U=(P^{\perp} + \widetilde{S}(t_{2,1}\otimes P))U.$ 
	
	Similarly, since $\mathcal{D}_{T^{(1)}} \oplus (E_1 \ot\mathcal{D}_{ T^{(2)}})\cong \mathcal{D}_{ T^{(2)}}\oplus (E_2\ot \mathcal{D}_{ T^{(1)}}),$  let $\mathcal{K}=\mathcal{D}_{T^{(1)}} \oplus(E_1 \ot\mathcal{D}_{T^{(2)}})$ and 
	$V(\Delta(T)h):=(\Delta(T^{(1)})h,  (I_{E_1} \ot \Delta(T^{(2)}))\wT^{(1)^*}h)$ for $h\in \mathcal{H}.$ 
	Since $$(I_{E_1}\ot V\Delta(T))\wT^{(1)^*}h=((I_{E_1}\ot \Delta(T^{(1)}))\wT^{(1)^*}h,(I_{E_1^{\ot 2}}\ot \Delta(T^{(2)}))\wT_2^{(1)^*}h), \quad  h\in \mathcal{H}.$$
	Let $i_1 : \mathcal{D}_{T^{(1)}} \raro \mathcal{K}$ and $i_2 :  E_1 \ot\mathcal{D}_{ T^{(2)}} \raro \mathcal{K}$ be
	the inclusion maps, defined by
	\[
	i_1(h_1) := (h_1,0)\quad  \mbox{and} \quad i_2 (h_2) := (0,h_2)
	\quad (h_1 \in \mathcal{D}_{T^{(1)}}, h_2 \in  E_1 \ot\mathcal{D}_{T^{(2)}}).
	\]
	Then $Q : = i_1 i_1^* \in B(\mathcal{K})$ is an orthogonal
	projection of $\mathcal{K}$ onto $ \mathcal{D}_{ T^{(1)}}$, i.e.,
	\[
	Q(h_1,h_2) = (h_1,0) \quad \quad \quad (h_1,h_2) \in \mathcal{K}.
	\]
	Thus $i_2 i_2^* = Q^\perp$ is the orthogonal projection of $\mathcal{K}$
	onto $E_1 \ot\mathcal{D}_{ T^{(2)}}.$ Define a
	unitary module map $U':E_1\ot\mathcal{K}\to E_1\ot E_2 \ot \mathcal{D}_{ T^{(1)}}\oplus E_1\ot \mathcal{D}_{ T^{(2)}}$ by
	\begin{align*}U'((I_{E_1}\ot \Delta(T^{(1)}))\wT^{(1)^*}h,(I_{E_1^{\ot 2}} \ot \Delta(T^{(2)}))\wT_2^{(1)^*}h):=\\((I_E\ot \Delta(T^{(1)}))\wT^{*}h,(I_{E_1}\ot \Delta(T^{(2)}))\wT^{(1)^*}h).\end{align*} 
	Then
	\[
	U_2=
	\begin{bmatrix}
		U'^*Q^{\perp}& U'^*(I_E\ot i_1)\\
		i_1^*& 0
	\end{bmatrix}: \mathcal{K} \oplus  (E \ot\mathcal{D}_{ T^{(1)}})\to (E_1 \ot \mathcal{K}) \oplus  \mathcal{D}_{ T^{(1)}}
	\] is a unitary module map. For $h\in \mathcal{H}$ we have
	\begin{align*}
		&U_2(V\Delta(T)h,(I_{E}\ot \Delta(T^{(1)}))\wT^{*}h)\\&=U_2(\Delta(T^{(1)})h,  (I_{E_1} \ot \Delta(T^{(2)}))\wT^{(1)^*}h,(I_{E}\ot \Delta(T^{(1)}))\wT^{*}h)\\&=(U'^*((I_E\ot \Delta(T^{(1)}))\wT^{*}h,(I_{E_1}\ot \Delta(T^{(2)}))\wT^{(1)^*}h),\Delta(T^{(1)})h)\\&=((I_{E_1}\ot \Delta(T^{(1)}))\wT^{(1)^*}h,(I_{E_1^{\ot 2}} \ot \Delta(T^{(2)}))\wT_2^{(1)^*}h,\Delta(T^{(1)})h)\\&=((I_{E_1}\ot V\Delta(T))\wT^{(1)^*}h,\Delta(T^{(1)})h).
	\end{align*} From Theorem \ref{identity}, there exists a contractive linear map $\widetilde{\Phi}_1:E_1 \ot \mathcal{K}\to \mathcal{F}(E)\ot \mathcal{K}$ such that $$(I_{E_1}\ot \Pi_{V})\wT^{(1)^*}=\widetilde{M}_{{\Phi}_1}^*\Pi_{V},$$ where $\widetilde{\Phi}_1=Q^{\perp}U'+(I_{E}\ot i_1i_1^*)U'=(Q^{\perp} + \widetilde{S}(I_E\otimes Q))U'.$ We conclude that $(\rho,M_{\Phi_1},M_{\Phi_2})$ is the pure isometric dilation of $(\sigma,T^{(1)},T^{(2)}).$ This complets the proof.
\end{proof}

We now prove the general setting, i.e., $\mbox{dim~} (\mathcal{D}_{T^{(1)}} \oplus (E_1 \ot\mathcal{D}_{ T^{(2)}}))= \infty ~~\mbox{or}~~~\mbox{dim~} (\mathcal{D}_{ T^{(2)}}\oplus (E_2\ot \mathcal{D}_{ T^{(1)}}))= \infty.$ The proof is largely identical to that of the previous theorem.

\begin{theorem}\label{DSAR4}
	Assume $(\sigma, T^{(1)}, T^{(2)})$ to be a pure c.c.c. representation of $\be$ over $\mathbb N^2_0$ on $\mathcal H$ such that $\mbox{dim~} (\mathcal{D}_{T^{(1)}} \oplus (E_1 \ot\mathcal{D}_{ T^{(2)}}))=\mbox{dim~} (\mathcal{D}_{ T^{(2)}}\oplus (E_2\ot \mathcal{D}_{ T^{(1)}}))= \infty,$ then $(\sigma, T^{(1)}, T^{(2)})$ dilates to a pure isometric covariant representation.
\end{theorem}
\begin{proof}
	Let $ \dim (\mathcal{D}_{T^{(1)}} \oplus  (E_1 \ot\mathcal{D}_{T^{(2)}}))$ or $\dim ((E_2 \ot \mathcal{D}_{T^{(1)}}) \oplus  \mathcal{D}_{T^{(2)}} )$ is infinite dimensional space, and let $\mathcal{D}$ be an infinite dimensional
	Hilbert space. Set $\mathcal K= \mathcal{D} \oplus \mathcal{D}_{T^{(2)}} \oplus(E_2 \ot\mathcal{D}_{T^{(1)}}),$ we define an isometry $V:\mathcal{D}_{T}\to  \mathcal K$ by
	$$V(\Delta(T)h):=(0,\Delta(T^{(2)})h,  (I_{E_2} \ot \Delta(T^{(1)}))\wT^{(2)^*}h)\quad \quad h\in \mathcal{H}.$$
	Let $j_1 : \mathcal{D}\oplus \mathcal{D}_{T^{(2)}} \raro \mathcal{K}$ and $j_2 :  E_2 \ot\mathcal{D}_{ T^{(1)}} \raro \mathcal{K}$ be
	the inclusion maps, defined by
	\[
	j_1(h,h_1) := (h,h_1,0)\quad  \mbox{and} \quad j_2 (h_2) := (0,0,h_2),
	\]
	for $h\in \mathcal{D},h_1 \in \mathcal{D}_{T^{(2)}}, h_2 \in  E_2 \ot\mathcal{D}_{ T^{(1)}}.$ Then $P= j_1 j_1^* \in B(\mathcal{K})$ is an orthogonal
	projection of $\mathcal{K}$ onto $\mathcal{D}\oplus \mathcal{D}_{ T^{(2)}}$, i.e.,
	\[
	P(h_1,h_2,h_3) = (h_1,h_2,0) \quad \quad \quad (h_1,h_2,h_3) \in \mathcal{K}.
	\]
	Thus, $j_2 j_2^* = P^\perp$ is the orthogonal projection of $\mathcal{K}$
	onto $E_2 \ot\mathcal{D}_{ T^{(1)}}.$
	Define an isometry
	\begin{align*}&
		U_{\mathcal{D}}: \{0_\mathcal{D}\}\oplus \{(I_{E_2} \ot \Delta(T^{(2)}))\wT^{(2)^*}h\oplus (I_{E_2^{\ot 2}} \ot \Delta(T^{(1)}))\wT_2^{(2)^*}h:~h\in \mathcal H\} \to \\& \:\:\:\:\:\:\:\: \{0_\mathcal{D}\}\oplus \{(t_{1,2}\ot \Delta(T^{(2)}))\wT^{*}h\oplus (I_{E_2}\ot \Delta(T^{(1)}))\wT^{(2)^*}h:~h\in \mathcal H\} 
	\end{align*} by 
	\begin{align*}&
		U_{\mathcal{D}}(0_\mathcal{D},(I_{E_2}\ot \Delta(T^{(2)}))\wT^{(2)^*}h,(I_{E_2^{\ot 2}} \ot \Delta(T^{(1)}))\wT_2^{(2)^*}h)\\&=(0_\mathcal{D},(t_{1,2}\ot \Delta(T^{(2)}))\wT^{*}h,(I_{E_2}\ot \Delta(T^{(1)}))\wT^{(2)^*}h).\end{align*}  $h\in \mathcal H.$ 
	It is easy to observe that $U_{\mathcal{D}}$ can be extended to a unitary operator from
	$\mathcal{D} \oplus E_2\ot \mathcal{D}_{ T^{(1)}}\oplus E_2^{\ot 2}\ot \mathcal{D}_{ T^{(1)}}$  to
	$\mathcal{D} \oplus E_2\ot E_1 \ot \mathcal{D}_{ T^{(2)}}\oplus E_2\ot \mathcal{D}_{ T^{(1)}}.$ Using above notations we define a unitary module map
	\[
	U_1=
	\begin{bmatrix}
		U_{\mathcal{D}}^*P^{\perp}& U_{\mathcal{D}}^*(t_{1,2}\ot j_1)\\
		j_1^*& 0
	\end{bmatrix}: \mathcal{K} \oplus  (E \ot (\mathcal{D} \oplus \mathcal{D}_{ T^{(2)}}))\to (E_2 \ot \mathcal{K}) \oplus  (\mathcal{D} \oplus \mathcal{D}_{ T^{(2)}}).
	\] such that 
	\begin{align*}
		U_1(V\Delta(T)h,(0,(I_{E}\ot \Delta(T^{(2)}))\wT^{*}h))=((I_{E_2}\ot V\Delta(T))\wT^{(2)^*}h,(0,\Delta(T^{(2)})h))
	\end{align*} for $h\in \mathcal{H}.$
	From Theorem \ref{identity}, there exists a contractive linear map $\widetilde{\Phi}_2:E_2 \ot \mathcal{K}\to \mathcal{F}(E)\ot \mathcal{K}$ such that $$(I_{E_2}\ot \Pi_V)\wT^{(2)^*}=\widetilde{M}_{{\Phi}_2}^*\Pi_V$$ where $\widetilde{\Phi}_2=P^{\perp}U+(I_{E}\ot j_1)(t_{2,1}\otimes j_1^*)U=(P^{\perp} + \widetilde{S}(t_{2,1}\otimes P))U.$
	
	Similarly, we get the relation for $\widetilde{\Phi}_1.$
	
\end{proof}

\section{Factorizations of completely contractive covariant representations}
Suppose that $(\sigma, T^{(1)}, T^{(2)})$ is a c.c.c. representation of $\be$ over $\mathbb N^2_0$ on a
Hilbert space $\mathcal H$ and $(\sigma, T)$ is a pure c.c.c. representation  of $E:=E_1\ot E_2$ on $\mathcal H,$ where $\widetilde{T}=\widetilde{T}^{(1)}(I_{E_1} \otimes \widetilde{T}^{(2)}).$ Then, by Theorem \ref{dilation}, we can observe $\wT\cong P_{\mathcal{Q}}\widetilde{S}|_{E\ot \mathcal{Q}},$ where $\mathcal{Q}=ran\Pi,$ $\Pi: \mathcal{H} \longrightarrow \mathcal{F}(E)\otimes\mathcal{D}_{T}$ is the isometric dilation of $(\sigma,T)$ and $(\rho,S)$ is the induced representation of $E$ on $\mathcal{F}(E)\otimes\mathcal{D}_{T}.$

Suppose that $\Pi_V=(I_{\mathcal{F}(E)}\ot V)\Pi: \mathcal{H} \longrightarrow \mathcal{F}(E)\otimes\mathcal{K}$ is the isometric dilation as in Theorems \ref{JJ3} and \ref{DSAR4}, that is, $$ \Pi_V \sigma(b)=\rho(b)\Pi_V,    \quad  \quad (I_{E_1}\ot \Pi_V)\wT^{(1)^*}=\widetilde{M}_{{\Phi}_1}^*\Pi_V$$ $$\quad and \quad (I_{E_2}\ot \Pi_V)\wT^{(2)^*}=\widetilde{M}_{{\Phi}_2}^*\Pi_V.$$ From Theorem \ref{JJ3}, we get $$(I_{E_1}\ot \Pi)\wT^{(1)^*}=(I_{E_1}\ot I_{\mathcal{F}(E)}\ot V^*)\widetilde{M}^*_{{\Phi}_1}(I_{\mathcal{F}(E)}\ot V)\Pi=\widetilde{M}_{\phi_1}^*\Pi,$$ where $\widetilde{\phi_1}=(I_{\mathcal{F}(E)}\ot V^*)\widetilde{\Phi}_1(I_{E_1}\ot V)$ and $V:\mathcal{D}_{T}\to  \mathcal K$ is the isometric module map. Similarly, $$(I_{E_2}\ot \Pi)\wT^{(2)^*}=(I_{E_2}\ot I_{\mathcal{F}(E)}\ot V^*)\widetilde{M}^*_{{\Phi}_2}(I_{\mathcal{F}(E)}\ot V)\Pi=\widetilde{M}_{\phi_2}^*\Pi,$$ where $\widetilde{\phi_2}=(I_{\mathcal{F}(E)}\ot V^*)\widetilde{{\Phi}}_2(I_{E_2}\ot V)$. In particular, $ran \Pi$ is co-invariant for $(\rho,M_{\phi_1},M_{\phi_2})$ of $\be$ on a
Hilbert space $\mathcal{F}(E)\ot \mathcal{D}_{T}.$ It follows that 
$$(I_{E} \ot \Pi)\widetilde{T}^{*}=\widetilde{S}^{*}\Pi,$$ and $ran \Pi$ is also a co-invariant for $(\rho,S)$ of $E$ on $\mathcal{F}(E)\ot \mathcal{D}_{T}.$

\begin{theorem}
	Let $(\sigma, T)$ be a pure c.c.c. representation  of $E$ on $\mathcal H$ and let $\wT\cong P_{\mathcal{Q}}\widetilde{S}|_{E\ot \mathcal{Q}}$ be the isometric dilation for $(\sigma, T)$ $($see Theorem \ref{dilation}$)$. If $\widetilde{T}=\widetilde{T}^{(1)}(I_{E_1} \otimes \widetilde{T}^{(2)})$ for some c.c.c. representation $(\sigma, T^{(1)}, T^{(2)})$ of $\be$ on $\mathcal H,$ then there exist completely bounded maps $\phi_1$ and $\phi_2$ such that $\mathcal{Q}$ is  co-invariant for $(\rho,M_{\phi_1},M_{\phi_2})$ of $\be$ on a
	Hilbert space $\mathcal{F}(E)\ot \mathcal{D}_{T}$ and $$(\sigma,\wT^{(1)},\wT^{(2)})\cong (\rho ,P_{\mathcal{Q}}\widetilde{M}_{\phi_1}|_{E_1\ot \mathcal{Q}},P_{\mathcal{Q}}\widetilde{M}_{\phi_2}|_{E_2\ot \mathcal{Q}}).$$ 
\end{theorem}
\begin{corollary}
	Let $(\sigma, T)$ be a pure c.c.c. representation  of $E$ on $\mathcal H$ and let $\wT\cong P_{\mathcal{Q}}\widetilde{S}|_{E\ot \mathcal{Q}}$ be the isometric dilation for $(\sigma, T)$ $($see Theorem \ref{dilation}$)$. Then $\widetilde{T}=\widetilde{T}^{(1)}(I_{E_1} \otimes \widetilde{T}^{(2)})$ for some c.c.c. representation $(\sigma, T^{(1)}, T^{(2)})$ of $\be$ on $\mathcal H$ if and only if there exist completely bounded linear maps $\phi_1$ and $\phi_2$ such that $\mathcal{Q}$ is  co-invariant for $(\rho,M_{\phi_1},M_{\phi_2})$ of $\be$ on $\mathcal{F}(E)\ot \mathcal{D}_{T},$ $$P_{\mathcal{Q}}\widetilde{S}|_{E\ot \mathcal{Q}}=P_{\mathcal{Q}}\widetilde{M}_{\phi_1}(I_{E_1}\ot \widetilde{M}_{\phi_2})|_{E\ot \mathcal{Q}}=P_{\mathcal{Q}}\widetilde{M}_{\phi_2}(I_{E_2}\ot \widetilde{M}_{\phi_1})(t_{1,2}\ot I_{\mathcal{F}(E)\ot \mathcal{D}_{T}})|_{E\ot \mathcal{Q}}.$$and $$(\sigma,\wT^{(1)},\wT^{(2)})\cong (\rho ,P_{\mathcal{Q}}\widetilde{M}_{\phi_1}|_{E_1\ot \mathcal{Q}},P_{\mathcal{Q}}\widetilde{M}_{\phi_2}|_{E_2\ot \mathcal{Q}}).$$
\end{corollary}

\section{Commuting pure contractions}

In this section, we study completely contractive covariant representations $(\sigma, T^{(1)}, T^{(2)})$ of $\be$ over $\mathbb N^2_0$ on a Hilbert space $\mathcal{H}$ that admit a canonical BCL triple associated with them. 

Let $(\sigma, T^{(1)}, T^{(2)})$ be a completely contractive covariant representation of $\be$ on $\mathcal H.$ From Equation (\ref{BB1}), we define an isometric module map 
\begin{align*}
	U:\{(I_{E_2} \ot \Delta(T^{(1)}))\wT^{(2)^*}h\oplus \Delta(T^{(2)})h: h\in \mathcal{H}\} \to\\ \{\Delta(T^{(1)})h\oplus (I_{E_1} \ot \Delta(T^{(2)}))\wT^{(1)^*}h: h\in \mathcal{H}\}
\end{align*} by
\begin{align}\label{JJ4}
\nonumber	U((I_{E_2} \ot \Delta(T^{(1)}))\wT^{(2)^*}h, \Delta(T^{(2)})h)=\\(\Delta(T^{(1)})h,  (I_{E_1} \ot \Delta(T^{(2)}))\wT^{(1)^*}h) \quad h\in \mathcal H. \end{align} Let us now introduce some subspaces
$$\mathcal{M}_{U}:=\{(I_{E_2} \ot \Delta(T^{(1)}))\wT^{(2)^*}h\oplus \Delta(T^{(2)})h: h\in \mathcal{H}\}\subseteq E_2\ot \mathcal{D}_{ T^{(1)}}\oplus \mathcal{D}_{ T^{(2)}},$$
$$\mathcal{N}_{U}:=\{ \Delta(T^{(1)})h \oplus (I_{E_1} \ot \Delta(T^{(2)}))\wT^{(1)^*}h : h\in \mathcal{H}\}\subseteq \mathcal{D}_{ T^{(1)}}\oplus E_1\ot \mathcal{D}_{ T^{(2)}},$$
Assume $P$ to be an orthogonal
projection of $\mathcal{D}_{ T^{(2)}} \oplus (E_2\ot \mathcal{D}_{ T^{(1)}})$ onto $ \mathcal{D}_{ T^{(2)}}$, i.e.,
\[
P(h_1,h_2) = (h_1,0) \quad \quad \quad (h_1,h_2) \in \mathcal{D}_{ T^{(2)}} \oplus (E_2\ot \mathcal{D}_{ T^{(1)}}).
\]
If we now assume that $\mbox{dim~} ((\mathcal{D}_{T^{(1)}} \oplus (E_1 \ot\mathcal{D}_{ T^{(2)}}))\ominus \mathcal{N}_{U})=\mbox{dim~} (((E_2\ot \mathcal{D}_{ T^{(1)}})\oplus \mathcal{D}_{ T^{(2)}})\ominus \mathcal{M}_{U}),$ then $U$ extends to a unitary (again denoted by $U$) from $E_2\ot \mathcal{D}_{ T^{(1)}}\oplus \mathcal{D}_{ T^{(2)}}$ to $\mathcal{D}_{ T^{(1)}}\oplus E_1\ot \mathcal{D}_{ T^{(2)}}.$ This discussion allows us to identify a canonical BCL triple associated to $(\sigma,T^{(1)}, T^{(2)}).$

\begin{definition}
	Let $(\sigma, T^{(1)}, T^{(2)})$ be a c.c.c. representation of $\be$ on $\mathcal H.$ Then a BCL triple $(\mathcal{D}_{ T^{(2)}} \oplus (E_2\ot \mathcal{D}_{ T^{(1)}}),U,P)$
	is said to be complete for $(\sigma, T^{(1)}, T^{(2)})$ if $\mbox{dim~} ((\mathcal{D}_{T^{(1)}} \oplus (E_1 \ot\mathcal{D}_{ T^{(2)}}))\ominus \mathcal{N}_{U})=\mbox{dim~} (((E_2\ot \mathcal{D}_{ T^{(1)}})\oplus \mathcal{D}_{ T^{(2)}})\ominus \mathcal{M}_{U}).$
\end{definition}
We have established in our discussion that the assumption: $(\mathcal{D}_{ T^{(1)}}\oplus E_1\ot \mathcal{D}_{ T^{(2)}},U,P)$ is a complete BCL triple for $(\sigma, T^{(1)}, T^{(2)})$ is equivalent to showing that there exists a unitary,
\begin{equation}\label{Dimple1}
	U = \begin{bmatrix}\widetilde{A}&B\\C&D\end{bmatrix}:
	(E_2 \ot\mathcal{D}_{ T^{(1)}})\oplus \mathcal{D}_{ T^{(2)}} \to \mathcal{D}_{ T^{(1)}}\oplus (E_1 \ot \mathcal{D}_{ T^{(2)}}),
\end{equation} where $(\sigma_1,A)$ is a completely contractive covariant representation of $E_2$ on $\mathcal{D}_{ T^{(1)}}$ define by $A(\xi)h=\widetilde{A}(\xi\ot h)$ for $\xi\in E_2, h\in \mathcal{D}_{ T^{(1)}}, \sigma_1=\sigma|_{\mathcal{D}_{ T^{(1)}}},$ $ B:\mathcal{D}_{ T^{(2)}}\to \mathcal{D}_{ T^{(1)}}, C:E_2\ot \mathcal{D}_{ T^{(1)}}\to E_1\ot \mathcal{D}_{ T^{(2)}}$ and $D:\mathcal{D}_{ T^{(2)}}\to E_1\ot \mathcal{D}_{ T^{(2)}}$ (see \cite{ARDS1}). For every $m\in \mathbb{N},$ define 
\begin{align*}
	D^{(m)}= (I_{E_1^{\otimes m-1}}\ot D)\dots(I_{E_1}\ot D)D :\mathcal{D}_{ T^{(2)}}\to E_1^{\ot m} \ot \mathcal{D}_{ T^{(2)}}.
\end{align*} Observe that $D^{(m+n)}=(I_{E_1^{\ot m}}\ot D^{(n)})D^{(m)}$ for all $m,n\in \mathbb{N}.$  The following identities based on Equation \ref{JJ4} are satisfied for all $h\in \mathcal{H},$
\begin{equation}\label{JJ5} 
	\Delta(T^{(1)})h= \widetilde{A}(I_{E_2} \ot \Delta(T^{(1)}))\wT^{(2)^*}h+
	B\Delta(T^{(2)})h,
\end{equation}
and \begin{equation}\label{JJ6} 
	(I_{E_1}\ot \Delta(T^{(2)}))\wT^{(1)^*}h= C(I_{E_2} \ot \Delta(T^{(1)}))\wT^{(2)^*}h+
	D\Delta(T^{(2)})h.
\end{equation} Note that with respect to the orthogonal projections $P,Q$ defined as above in Theorem \ref{JJ3}, we have
\begin{equation}
	U = \begin{bmatrix}\widetilde{A}&B\\C&D\end{bmatrix}=\begin{bmatrix}QUP^{\perp}|_{ran{P^{\perp}}}& QUP|_{ran{P}}\\Q^{\perp}U P^{\perp}|_{ran{ P^{\perp}}}&Q^{\perp}UP|_{ran{P}}\end{bmatrix},
\end{equation}
Since $U^*U = I = UU^*,$ we get
\begin{align*}
	\begin{bmatrix}\widetilde{A}^*\widetilde{A}+C^*C&\widetilde{A}^*B+C^*D\\B^*\widetilde{A}+D^*C&B^*B+D^*D\end{bmatrix}=\begin{bmatrix}I&0\\0&I\end{bmatrix}=\begin{bmatrix}\widetilde{A}\widetilde{A}^*+BB^*&\widetilde{A}C^*+BD^*\\C\widetilde{A}^*+DB^*&CC^*+DD^*\end{bmatrix}.
\end{align*}
If we act by $\widetilde{A}^*$ on the left of Equation (\ref{JJ5}), we have
\begin{align*}
	&\widetilde{A}^*\Delta(T^{(1)})h= \widetilde{A}^*\widetilde{A}(I_{E_2} \ot \Delta(T^{(1)}))\wT^{(2)^*}h+
	\widetilde{A}^*B\Delta(T^{(2)})h\\&=(I-C^*C)(I_{E_2} \ot \Delta(T^{(1)}))\wT^{(2)^*}h-C^*D\Delta(T^{(2)})h\\&=(I_{E_2} \ot \Delta(T^{(1)}))\wT^{(2)^*}h-C^*(C(I_{E_2} \ot \Delta(T^{(1)}))\wT^{(2)^*}h+
	D\Delta(T^{(2)})h)\\&=(I_{E_2} \ot \Delta(T^{(1)}))\wT^{(2)^*}h-C^*(I_{E_1}\ot \Delta(T^{(2)}))\wT^{(1)^*}h.
\end{align*} Similarly, if we act by $D^*$ on the left of Equation (\ref{JJ6}), we get
\begin{align*}
	&D^*(I_{E_1}\ot \Delta(T^{(2)}))\wT^{(1)^*}h\\&= D^*C(I_{E_2} \ot \Delta(T^{(1)}))\wT^{(2)^*}h+
	D^*D\Delta(T^{(2)})h \\&=-B^*\widetilde{A}(I_{E_2} \ot \Delta(T^{(1)}))\wT^{(2)^*}h+(I-B^*B)\Delta(T^{(2)})h\\&=\Delta(T^{(2)})h-B^*(\widetilde{A}(I_{E_2} \ot \Delta(T^{(1)}))\wT^{(2)^*}h+
	B\Delta(T^{(2)})h)\\&=\Delta(T^{(2)})h-B^*\Delta(T^{(1)})h.
\end{align*} Therefore, we obtain
\begin{equation}\label{JJ7}
	\Delta(T^{(2)})h=D^*(I_{E_1}\ot \Delta(T^{(2)}))\wT^{(1)^*}h+B^*\Delta(T^{(1)})h
\end{equation}
and 
\begin{equation}\label{JJ8}
	(I_{E_2} \ot \Delta(T^{(1)}))\wT^{(2)^*}h=	\widetilde{A}^*\Delta(T^{(1)})h+C^*(I_{E_1}\ot \Delta(T^{(2)}))\wT^{(1)^*}h.
\end{equation}

We now derive a recurrence relation connecting the entries of $U$ and $(\sigma, T^{(1)}, T^{(2)}).$

\begin{lemma}\label{JJ9}
	For all $m\in\mathbb{N}$ and $h\in \mathcal{H},$ we have
	\begin{enumerate}
		\item $\widetilde{A}_m^*\Delta(T^{(1)})h = (I_{E_2^{\ot m}}\ot \Delta(T^{(1)}))\wT_m^{(2)^*}h - \sum_{k=0}^{m-1} (I_{E_2^{\ot m-1-k}}\ot (I_{E_2}\ot \widetilde{A}_k^*)C^*(I_{E_1}\ot \Delta(T^{(2)})) \\\wT^{(1)^*})\wT_{m-1-k}^{(2)^*}h;$
		
		\item $D^{(m)}\Delta(T^{(2)})h=(I_{E_1^{\ot m}}\ot \Delta(T^{(2)}))\wT_m^{(1)^*}h - \sum_{k=0}^{m-1} (I_{E_1^{\ot m-1-k}}\ot (I_{E_1}\ot D^{(k)})C(I_{E_2}\ot \Delta(T^{(1)}))\wT^{(2)^*})\wT_{m-1-k}^{(1)^*}h.$
	\end{enumerate}
\end{lemma}
\begin{proof}
	We shall prove by the mathematical induction. It is clear from Equation (\ref{JJ8}) that $(1)$ is satisfied for $m=1$ and from Equation (\ref{JJ8}), we get
	\begin{align*}
		&(I_{E_2^{\ot m-1}}\ot \widetilde{A}^*\Delta(T^{(1)}))\wT_{m-1}^{(2)^*}h=\\&(I_{E_2^{\ot m}}\ot \Delta(T^{(1)}))\wT_{m}^{(2)^*}h-(I_{E_2^{\ot m-1}}\ot C^*(I_{E_1}\ot \Delta(T^{(2)}))T^{(1)^*})\wT_{m-1}^{(2)^*}h.
	\end{align*}
	Now, let us assume that $(1)$ is satisfied for $m=n-1$ for some $n\in \mathbb{N}$ that is,
	\begin{align*}
		\widetilde{A}_{n-1}^*\Delta(T^{(1)})h&=(I_{E_2^{\ot {n-1}}}\ot \Delta(T^{(1)}))\wT_{n-1}^{(2)^*}h \\&- \sum_{k=0}^{n-2} (I_{E_2^{\ot n-2-k}}\ot (I_{E_2}\ot \widetilde{A}_k^*)C^*(I_{E_1}\ot \Delta(T^{(2)}))\wT^{(1)^*})\wT_{n-2-k}^{(2)^*}h.
	\end{align*}
	Then, for $m = n$ we have
	\begin{align*}
		&\widetilde{A}_n^*\Delta(T^{(1)})h=(I_{E_2^{\ot n-1}}\ot \widetilde{A}^*)\widetilde{A}_{n-1}^*\Delta(T^{(1)})h\\&=(I_{E_2^{\ot n-1}}\ot \widetilde{A}^*)(I_{E_2^{\ot {n-1}}}\ot \Delta(T^{(1)}))\wT_{n-1}^{(2)^*}h  \\& \quad \quad - \sum_{k=0}^{n-2} (I_{E_2^{\ot n-1}}\ot \widetilde{A}^*)(I_{E_2^{\ot n-2-k}}\ot (I_{E_2}\ot \widetilde{A}_k^*)C^*(I_{E_1}\ot \Delta(T^{(2)}))\wT^{(1)^*})\wT_{n-2-k}^{(2)^*}h\\&=(I_{E_2^{\ot {n-1}}}\ot \widetilde{A}^*\Delta(T^{(1)}))\wT_{n-1}^{(2)^*}h \\&\quad\quad- \sum_{k=0}^{n-2} (I_{E_2^{\ot n-2-k}}\ot (I_{E_2}\ot \widetilde{A}_{k+1}^*)C^*(I_{E_1}\ot \Delta(T^{(2)}))\wT^{(1)^*})\wT_{n-2-k}^{(2)^*}h\\&=(I_{E_2^{\ot n}}\ot \Delta(T^{(1)}))\wT_{n}^{(2)^*}h-(I_{E_2^{\ot n-1}}\ot C^*(I_{E_1}\ot \Delta(T^{(2)}))\wT^{(1)^*})\wT_{n-1}^{(2)^*}h \\& \quad \quad - \sum_{k=0}^{n-2} (I_{E_2^{\ot n-2-k}}\ot (I_{E_2}\ot \widetilde{A}_{k+1}^*)C^*(I_{E_1}\ot \Delta(T^{(2)}))\wT^{(1)^*})\wT_{n-2-k}^{(2)^*}h\\&=(I_{E_2^{\ot n}}\ot \Delta(T^{(1)}))\wT_{n}^{(2)^*}h \\&\quad \quad-\sum_{k=0}^{n-1} (I_{E_2^{\ot n-1-k}}\ot (I_{E_2}\ot \widetilde{A}_{k}^*)C^*(I_{E_1}\ot \Delta(T^{(2)}))\wT^{(1)^*})\wT_{n-1-k}^{(2)^*}h.
	\end{align*} Hence $(1)$ holds for $m\in \mathbb{N}.$ Statement $(2)$ follows in a similar manner by using Equation (\ref{JJ6}). 
\end{proof}
Let $(\sigma, T)$ be a c.c.c. representation of $E$ on $\mathcal H$, we define a subspace $$H_i(\wT):=\{h\in \mathcal{H}: \|\wT_{n}(\xi_n\ot h)\|=\|\xi_n\ot h\| ~~\mbox{for}~~ n\in\mathbb{N}_0,\xi_n\in E^{\ot n}\}.$$ In particular, $H_i(\widetilde{A}^*)=\{h\in \mathcal{D}_{ T^{(1)}}: \| \widetilde{A}_m^*h\|=\|h\| ~~\mbox{for}~~ n\in\mathbb{N}_0\}$ and $H_i(D)=\{h\in \mathcal{D}_{ T^{(2)}}: \| D^{(n)}h\|=\|h\| ~~\mbox{for}~~ n\in\mathbb{N}_0\},$ where $\widetilde{A}$ and $D$ defined as above in Equation (\ref{Dimple1}).

We are now ready to obtain our main result on the interplay between pairs of commuting pure contractions and their corresponding BCL triples.
\begin{theorem}\label{JJ10}
	Assume $(\sigma, T^{(1)}, T^{(2)})$ to be a c.c.c. representation of $\be$ on $\mathcal H$ with each $(\sigma,T^{(i)}),i\in I_2$ to be pure and $(\mathcal{D}_{ T^{(2)}} \oplus (E_2\ot \mathcal{D}_{ T^{(1)}}),U,P)$ to be a complete BCL triple for $(\sigma, T^{(1)}, T^{(2)}).$ Then the subspaces $H_i(\widetilde{A}^*)=H_i(D)=\{0\}.$
\end{theorem}
\begin{proof}
	Let us begin by observing that $U$ is a unitary operator implies that $\widetilde{A}^*=P^{\perp}U^*Q|_{ran{Q}}$ is a contraction. Moreover, $\|\widetilde{A}_m^*h\|^2=\|\widetilde{A}_{m-1}^*h\|^2-\|(I_{E_2^{\ot m-1}}\ot B^*)\widetilde{A}_{m-1}^*h\|^2$ for all $h\in \mathcal{D}_{ T^{(1)}}$ and $m\in \mathbb{N}.$ Let $h\in H_i(\widetilde{A}^*),$ then $\| \widetilde{A}_m^*h\|=\|h\|$ for all $m\in\mathbb{N}_0,$ and hence $\|(I_{E_2^{\ot m-1}}\ot B^*)\widetilde{A}_{m-1}^*h\|=0.$
	More precisely, $H_i(\widetilde{A}^*)$ is a $\widetilde{A}^*$-invariant subspace of $ker (I_{E_2^{\ot m-1}}\ot B^*).$ 
	For $h\in H_i(\widetilde{A}^*),\Delta(T^{(1)})h'\in \mathcal{D}_{ T^{(1)}}$ and $m\in\mathbb{N},$ we have
	\begin{align*}
		&\langle h,\Delta(T^{(1)})h'\rangle=	\langle \widetilde{A}_m\widetilde{A}_m^*h,\Delta(T^{(1)})h'\rangle=\langle \widetilde{A}_m^*h,\widetilde{A}_m^*\Delta(T^{(1)})h'\rangle\\&=\langle \widetilde{A}_m^*h,(I_{E_2^{\ot m}}\ot \Delta(T^{(1)}))\wT_m^{(2)^*}h'\rangle- \\& \quad \quad\sum_{k=0}^{m-1} \langle \widetilde{A}_m^*h,(I_{E_2^{\ot m-1-k}}\ot (I_{E_2}\ot \widetilde{A}_k^*)C^*(I_{E_1}\ot \Delta(T^{(2)}))\wT^{(1)^*})\wT_{m-1-k}^{(2)^*}h'\rangle \\&=\langle \widetilde{A}_m^*h,(I_{E_2^{\ot m}}\ot \Delta(T^{(1)}))\wT_m^{(2)^*}h'\rangle- \\& \quad \quad\sum_{k=0}^{m-1} \langle (I_{E_2^{\ot m-1-k}}\ot C(I_{E_2}\ot \widetilde{A}_k))\widetilde{A}_m^*h,\\& \quad \quad \:\:\:\:\:\:(I_{E_2^{\ot m-1-k}}\ot(I_{E_1}\ot \Delta(T^{(2)}))\wT^{(1)^*})\wT_{m-1-k}^{(2)^*}h'\rangle \\&=\langle \widetilde{A}_m^*h,(I_{E_2^{\ot m}}\ot \Delta(T^{(1)}))\wT_m^{(2)^*}h'\rangle- \\& \quad \quad\sum_{k=0}^{m-1} \langle (I_{E_2^{\ot m-1-k}}\ot C\widetilde{A}^*)\widetilde{A}_{m-k-1}^*h,\\& \quad \quad \:\:\:\:\:\:(I_{E_2^{\ot m-1-k}}\ot(I_{E_1}\ot \Delta(T^{(2)}))\wT^{(1)^*})\wT_{m-1-k}^{(2)^*}h'\rangle \\&=\langle \widetilde{A}_m^*h,(I_{E_2^{\ot m}}\ot \Delta(T^{(1)}))\wT_m^{(2)^*}h'\rangle+ \\& \quad \quad\sum_{k=0}^{m-1} \langle (I_{E_2^{\ot m-1-k}}\ot DB^*)\widetilde{A}_{m-k-1}^*h,\\& \quad \quad \:\:\:\:\:\:(I_{E_2^{\ot m-1-k}}\ot(I_{E_1}\ot \Delta(T^{(2)}))\wT^{(1)^*})\wT_{m-1-k}^{(2)^*}h'\rangle \\&=\langle \widetilde{A}_m^*h,(I_{E_2^{\ot m}}\ot \Delta(T^{(1)}))\wT_m^{(2)^*}h'\rangle.
	\end{align*} This implies that $|\langle h,\Delta(T^{(1)})h'\rangle|\le \lim_{m\to \infty} \|h\|\|\wT_m^{(2)^*}h'\|=0,$ and hence $H_i(\widetilde{A}^*)\perp \mathcal{D}_{ T^{(1)}}.$ This is a contradiction to the fact that $H_i(\widetilde{A}^*)\subseteq \mathcal{D}_{ T^{(1)}},$
	unless $H_i(\widetilde{A}^*)=0.$ Similarly, from statement $(2)$ of Lemma \ref{JJ9} we get $H_i(D)=\{0\}.$ This completes the proof. 
\end{proof}

\subsection*{Conflict of Interest}
We state that there is no conflict of interest and we have no personal relationships that could have appeared to influence the work reported in this paper.

\subsection*{Data Availability}
Data sharing is not applicable to this article as no datasets were generated or analyzed during the current study.

\subsection*{Acknowledgement}
The author thanks the reviewers for carefully reading the manuscript and the editor for suggesting changes in the manuscript. The author want to thank Azad Rohilla for some fruitful discussions. The author declares that no external funding was received for this research.

\end{document}